\documentclass{amsart}

\usepackage{amssymb}
\usepackage{mathabx}
\usepackage{mathrsfs}
\usepackage{enumitem}

\usepackage{graphicx}

\makeatletter
\@namedef{subjclassname@2010}{%
  \textup{2010} Mathematics Subject Classification}
\makeatother

\usepackage[T1]{fontenc}

\newtheorem{theorem}{Theorem}[section]

\newtheorem{proposition}[theorem]{Proposition}

\theoremstyle{definition}

\newtheorem{example}[theorem]{Example}

\renewcommand{\a}{\alpha}

\newcommand{\e}{\epsilon}

\newcommand{\G}{\Gamma}

\renewcommand{\o}{\omega}
\renewcommand{\O}{\Omega}
\newcommand{\p}{\pi}
\newcommand{\ph}{\phi}
\newcommand{\pa}{\partial}
\renewcommand{\r}{\rho}
\newcommand{\Sig}{\Sigma}
\newcommand{\sig}{\sigma}

\newcommand{\we}{\wedge}

\numberwithin{equation}{section}

\begin{document}




\title[Maximal  domains  for analytic quasilinear PDEs ]{Maximal domains of solutions  for  analytic quasilinear  differential equations of first order  }

\author[C.-K. Han]{Chong-Kyu Han}
\address[C.-K. Han]{Research institute of Mathematics
\\ Seoul National University\\ 1 Gwanak-ro, Gwanak-gu\\ Seoul 08826, Republic of Korea}
\email{ckhan@snu.ac.kr}
%

\author[T. Kim]{Taejung Kim}
\address[T. Kim]{Department of Mathematical Education\\ Korea National University of Education\\ 250 Taeseongtabyeon-ro,  Gangnae-myeon\\ Heungdeok-gu,  Cheongju-si \\ Chungbuk 28173, Republic of Korea}
\email{tjkim@kias.re.kr}

\date{\today}

\thanks{The  authors were supported by  NRF-Republic of Korea with  grants
 0450-20210049 (C.-K. Han) and  2018R1D1A3B07043346 (T. Kim), respectively.}

\begin{abstract}  We study the  real-analytic continuation of  local real-analytic solutions to the Cauchy problems of  quasi-linear partial differential equations of first order for a scalar function.    By making use of the first integrals of the characteristic vector field and the implicit function theorem we  determine the maximal domain of the analytic extension of a local solution as a single-valued function.  We present some examples including the scalar conservation laws that admit global first integrals so that our method is applicable.
\end{abstract}



\maketitle

\section{Introduction}\label{intro}
An important feature of  analytic functions,  either  in  complex or real variables,   is that a germ of a function determines the  function globally.   In complex analysis, special functions like the Riemann zeta function or the gamma function are defined  by defining equations in part of the complex plane  and then extended by analytic continuation.  For those two special functions the maximal domain is
 the whole complex plane except for the poles. However, for  functions like  $ \log z $ or $\sqrt{z} $    analytic continuation leads to a multi-valued function so that the  domain of maximal analytic continuation is a Riemann surface that  cannot be  embedded  in the complex plane.
We shall say there is no maximal domain  for  $\log z $ or for $\sqrt{z}$.    A domain $U\subset \Bbb C$ is  the maximal domain of an analytic function $f$ if $f$ cannot be analytically continued across any of the boundary point of $U$.
It is not difficult to show that any domain in $\Bbb C$ is a maximal domain for some complex analytic function.  As for an analytic function of several complex variables a domain  being maximal is the  notion of domain of holomorphy.

We study  in this paper how to determine  the maximal domains for real-analytic  functions   given as  solutions of  differential equations with analytic  data.  This problem seems  to be difficult in general.   However, for  quasi-linear partial differential equations that have globally defined first integrals one can decide the maximal domain  by applying the implicit function theorem to the  first integrals determined by the initial data  (Theorem \ref{main}).

A function $f(x)$  in $n$ real variables $ x=(x_1, \ldots, x_n)   $  defined on a connected
 open set $U\subset \Bbb R^n$   is said to be analytic if
at every point $p = (p_1, \ldots, p_n) \in U,  $  $~f(x) $  is representable as a  convergent power series in
$(x-p) := (x_1-p_1, \ldots, x_n-p_n),  $
   that is, $f(x)$ is locally equal to its Taylor series.
  We shall denote by $ C^{\o}(U)$  the set of analytic functions in $U$.
A function $ f_1 \in C^{\o}(U_1)$,  where $U_1$ is a connected open set with $U\cap U_1 \neq \emptyset$,
is called a direct analytic continuation  of $f$  if $f(x)= f_1(x)  $  for all  $x\in U\cap U_1$.
If $U_1 \supset U, $ the direct analytic continuation $f_1$ shall be called an {\it analytic extension} of $f$.   Now
 consider a sequence   $f_k \in C^{\o}(U_k), $  $k=1, \ldots, N,$  where $f_k $ is a direct analytic continuation of $f_{k-1}$ such that   $U\cap U_N\neq\emptyset. $    For any point   $x\in U\cap U_N,$
 $f_N(x)$ need not be the same as $f(x)$  as we see in the following example.  Thus in general, a sequence of direct analytic continuations yields a multi-valued  function.

\begin{example}
  Let $U\subset \Bbb R^2 $ be an open disk of radius $1/2$ centered at $(1,0) $ and
let
$f(x_1, x_2):= \tan^{-1} (x_2/x_1).$   Then  $f_N(x) $ and $f(x)$  may  differ by a multiple of $2\p$.
\end{example}

A boundary point $p$ of $U$ is said to be regular if $p$ has a neighborhood to which $ f$ continues analytically,
that is, there is a neighborhood $N(p)$ of $p$ and a direct analytic continuation $f_1 \in C^{\o}(N(p)) $ of $f$.
If a boundary point of $U$ is not regular it is said to be singular.    If every boundary point of $U$ is singular  then  $U$ is the maximal
domain of  extension and  $f$  is maximally extended as a single-valued function.

  We observe that for an analytic differential equation
 an analytic extension of a local analytic solution is also a solution, thus the maximal extension gives  the global solution (Theorem~\ref{maximal solution}).

A differential equation is said to be quasi-linear if it is linear in the highest order
 derivatives of the unknown function.   Thus a   quasi-linear PDE
of first order  for $u(t,x)$, $~x=(x_1, \cdots, x_n),~$ can be written as
\begin{equation}\label{pde}
\a(t,x,u) u_t~  + ~\sum_{k=1}^n a_k(t,x,u) u_{x_k} = b(t, x,u),
\end{equation}  where $  u_t := \frac{\partial u}{\partial t}, ~ u_{x_k} := \frac{\partial u}{\partial x_k}.  ~$   We assume  $\a$ is nowhere vanishing and
find a  solution
subject to the initial condition
\begin{equation}\label{data}
u(0, x) = h(x), \quad  | x| < \epsilon
\end{equation}  for an arbitrarily  small  $\e>0.$
We shall call
\begin{equation}\label{gamma1}\Gamma:=\{(0,x,h(x)):   |x| < \epsilon \}
\end{equation}  the initial  set.
We assume  that all the coefficients $\a,$ $a_k$ and $b$ are real-analytic ($C^{\o}$)  in a connected open set
$\O\subset \Bbb R^{n+2} = \{(t,x,u)\} $ that contains $\G.$
We also assume  that
$h(x)$ is $C^{\o}$  in a small ball $|x|<\e. $
Then by the Cauchy-Kowalevski theorem there is a unique $C^{\o}$  solution $u(t,x) $ on a neighborhood of the origin of
$\mathbb R^{n+1}=\{(t,x)\}.$

The Cauchy-Kowalevski theorem is an existence theory that is applicable to a wider class of analytic differential equations with analytic initial data. But for  quasi-linear PDEs of first order, one can prove the existence of $C^{\o}$ solution on an open  neighborhood
of
$\{(0, x) : |x| < \e\}$ in $\mathbb R^{n+1} $ by the method of characteristics.
By the uniqueness of analytic solutions the solution obtained by  the method of characteristics and the
one
by the Cauchy-Kowalevski theorem are equal in their common domain.  Then the uniqueness of the analytic extension  implies the followings,  which are
  rather surprising.

\begin{itemize}
\item[i)] An analytic extension  $\hat u(t,x) $ of this local solution $u(t,x)$ satisfies (\ref{pde}).

\item[ii)] Its initial value $\hat{u} (0,x)$ is  an  analytic extension   of  $h(x) $ if the set $\{(0,x)\in \text{domain of }\hat{u}\} $ is connected.
\end{itemize}

 Suppose that a local solution to  (\ref{pde})-(\ref{data})  is  given implicitly as
\begin{equation} \label{simultaneous}  {F}(t,x,u) = 0, \end{equation}
  where $F$ is  analytic in $\O$,   vanishes on  $\G , $  and satisfies the nondegeneracy condition
\begin{equation}\label{simultaneous20} {F}_u (t,x,u)  \neq 0, \end{equation}
so that one obtains an  explicit solution $u(t,x)$   by the implicit function theorem.  Then the  zero locus of $F$ is locally the graph of the  explicit solution $u(t,x).$
Now let  $\Sig$ be  the set defined by    (\ref{simultaneous})  and $\sigma \subset \Sigma$ be the subset given by
\begin{equation}\label{sigma^n}  {F}_u (t,x,u)  =  0. \end{equation}
If needed we specify the  dimensions by superscripts as   in $\sig^n,  $  $ \Sig^{n+1},  $ and $\G^n, $ respectively.
Let $\pi $ be the projection $ (t,x,u) \stackrel{\pi}{\rightarrow} (t,x).$  We shall show in \S~\ref{psec3}  that
the  connected component of $\Sigma \setminus \sigma $ that contains $\Gamma$
is the graph of the maximally extended solution so that its image under $\pi$  is the  maximal  domain of analytic extension  of the local solutions to (\ref{pde})-(\ref{data})
assuming that  the characteristic vector field of (\ref{pde}) admits $n+1$ first integrals that are defined globally on $\O$ (see   \S~\ref{psec2} and \S~\ref{psec3} for definitions).
In particular, our method is  useful for scalar conservation laws, where (\ref{pde}) has the form
\begin{equation}\label{cons law general}
u_t + \sum_{k=1}^n a_k(u) u_{x_k} = 0.
\end{equation}
The functions
$$ \begin{aligned} \r_1 & := u \\
\r_k &:= x_k - a_k(u) t, ~~\quad k=1, \cdots, n
\end{aligned} $$
are the first integrals of (\ref{cons law general}) that are defined globally.


\section{Preliminaries }\label{psec2}
\subsection{Real analyticity  }   We review some basic facts on real analytic functions.
\vskip 1pc
\begin{theorem}\label{identity thm} {\text (Identity theorem)}
Suppose that  $V\subset \Bbb R^n, $ for any positive integer $n$,  is a connected open set  and  $f\in C^{\o}(V)$  is identically  zero on a small open ball that is contained in $V. $  Then  $f$ is identically zero on $V$.
\end{theorem}  Theorem \ref{identity thm}  implies the following

\begin{theorem} \label{unique} (Uniqueness of the analytic extension)
Let  $U$  and $V$ are connected open subsets of $\Bbb R^n$ with    $U\subset V  $ and let
 $f\in C^{\o}(U). $  Suppose that
$F_1, F_2 \in C^{\o}(V)$   are  extensions   of $f. $   Then $F_1=F_2.$
\end{theorem}

The implicit function theorem states that if $ F(x,u) $ is a smooth function in the variables   $x=(x_1, \cdots, x_n)  $
and $u$,  and if
$F_u \neq 0 $ at a point $(a, b)$, $a=(a_1, \cdots, a_n), $  then   $F(x,u)=0 $ is solvable for $u$  as a  function
of $x, $ namely,  there is a function $f(x)$ with $f(a) = b, $  such that $F(x,f(x)) = 0, $ for all $x$ in a neighborhood of $a.$
The analytic implicit function theorem states  that if $F$ is $C^{\o}$ then $f$ is $C^{\o}.$  More precisely,

\begin{theorem}\label{implicit function} (Analytic implicit function theorem)   Let $F$ be $C^{\o}$ on an open subset of $\Bbb R^{n+1} = \{(x,u)\}$, $x=(x_1, \ldots, x_n).$  Suppose that $F (x, u)$ is  $C^{\o}$ in its arguments
and that  $F_u (a,b) \neq 0.$  Then there exists an open neighborhood  $U\subset  \Bbb R^n$ of $a$  and  $f\in C^{\o}(U)$  with
$f(a)=b $  such that
$$  F(x, f(x))=0, \quad \forall x\in U.$$
\end{theorem}
One can prove Theorem \ref{implicit function} by the method of majorants for the power series expansion of $F$:
\begin{equation*}\label{majorants}  F(x,u) = \sum _{\a, k} a_{\a, k} x^{\a} u^k , \end{equation*}  where $\a = (\a_1, \cdots, \a_n) $ is a  multi-index and $k $ is  a non-negative integer.   Recall that  the Cauchy-Kowalevski theorem is proved by the method of the majorant.    Theorem \ref{implicit function} can be also proved  by an application of the Cauchy-Kowalevski theorem.  We refer the readers to  \cite{Kr} for the proofs.

\vskip 1pc
Proofs for the existence and uniqueness of solutions of  ODEs are based on the convergence of the iteration of the integral operator associated to the differential equation.  An analytic  version is the following
\begin{theorem}\label{fundamental} (Existence theorem for $C^{\o}$ ODEs)  Let $ g = (g_1, \cdots, g_n) $ be a system of $C^{\o}$ functions in $(t,x) \in U \subset \Bbb R^{n+1}$, $ x=(x_1, \cdots, x_n).$ Then for the initial value problem
\begin{equation*} \frac{dx}{dt} = g(t,x),   \quad  x(0) = x_0,  \end{equation*} for any $x_0$ with $(0,x_0)\in U, $
  there exists a unique system of  analytic  functions $\phi (t) = (\ph_1(t), \cdots, \ph_n(t))$ that satisfies
  $\phi(0)=x_0 $ and
  \begin{equation*}\label{analytic ode}
  \frac{d\phi}{dt} = g(t, \phi(t)), \quad |t|<\e, ~ \text{ for some } \e>0.
  \end{equation*}
\end{theorem}
One can prove Theorem \ref{fundamental} by complexifying the variables $t $ and $x$ and using the fact that a uniform limit of complex analytic functions is complex analytic (see \cite{CL}).

\subsection{Analytic extension  of local solutions}
Let us consider a system of  analytic differential equations in its most general setting
\begin{equation}\label{most general} \triangle (x,u^{(m)}) = 0 \end{equation}
where $\triangle = (\triangle_1, \cdots, \triangle_{\ell})$ is a system of $\ell$ partial differential equations of order $m$ for  unknown functions $u=(u^1,  \cdots, u^q) $ in $n$  independent variables  $x=(x_1, \cdots, x_n) \in \Omega\subset \Bbb R^n,  $  and $u^{(m)} $ is  the partial derivatives of $u$ of order up to $m$.  We assume that  $\Omega $ is a connected open set.   Our basic observation is the following

\begin{theorem}  \label{maximal solution}
Suppose (\ref{most general}) is analytic,  that is,  $\triangle$ is  $C^{\o}$ in its arguments,  and
that an analytic function $u=f(x)$ defined on a small neighborhood of $p\in \O$  satisfies  (\ref{most general}).  Let  ${\hat f }(x) $  be  an analytic extension   of $f(x)$  to  $\O. $    Then $\hat f (x)$ satisfies (\ref{most general}).
\end{theorem}

\begin{proof} On a small neighborhood of $p,$
\begin{equation*}\label{proof eq} 0 = \triangle(x,f^{(m)}(x)) = \triangle(x,\hat{f}^{(m)}(x)).\end{equation*} Hence by the identity theorem   $\triangle (x,\hat f^{(m)}(x)) = 0 $ for all $x\in \O.$
\end{proof}

\subsection{First integrals and invariant submanifolds}
   In this subsection we define the notions of first integral and invariant submanifold in smooth ($C^{\infty})$ category and from  the local  viewpoint, namely,   our functions, vector fields, and submanifolds  are in $C^{\infty}$-category  and defined on a small open set of $\Bbb R^N.$   Thanks to Theorem \ref{fundamental},   the Cauchy-Kowalevski theorem, and other basic facts on  analytic functions,  all the statements in this subsection hold true in analytic category as well.

 A system of  $C^{\infty}$  real-valued functions $\vec{\r} :=(\r_1, \ldots, \r_d)$ that are defined on an open subset  $U\subset \Bbb R^N$ is said to be {\it non-degenerate} if
\begin{equation*}\label{nondeg} d\r_1\we \cdots \we d\r_d \neq 0.\end{equation*}  Then their common zero set ${\vec \r} = 0 $ is a smooth submanifold of $U$ of codimension $d$.
Conversely, a submanifold of codimension $d$ is locally the common zero-set of a  non-degenerate system of  $d$ real-valued functions
${\vec \r} = (\r_1, \ldots, \r_d),$  which we call  local defining functions of the submanifold.

 Given a $C^{\infty}$  nowhere vanishing vector field
\begin{equation}\label{vf} X = \sum_{j=1}^N a_j (x) \frac{\pa}{\pa x_j} \end{equation}    with $C^{\infty}$  coefficients $a_j $ defined on an open subset of $\mathbb R^N =\{(x_1, \ldots, x_N)\}, $    a real-valued function $\r$ is called a {\it{first integral}} of $X$ if $\r$ is invariant under the flow of $X$, that is,  if
\begin{equation*}  X\r = 0.\end{equation*}

\begin{proposition}  Suppose that $X$ is a $ C^{\infty} $ nowhere vanishing vector field defined on an open subset of $\Bbb R^N$. Then

\begin{itemize}

\item[i)] there exists locally a non-degenerate set of $N-1$ first integrals
\begin{equation*} \label{nondeg first} {\vec \r}  = (\r_1, \cdots, \r_{N-1}) .
\end{equation*}

\item[ii)]  A $C^{\infty}$  function $F$  is invariant under the flow of $X$  if and only if
$F$ is a function of $\r_j$'s, that is,  there exists a $C^{\infty}$  function $f$ in $N-1$ variables so that
\begin{equation*} F = f\circ \vec{\r}.
\end{equation*}
\end{itemize}
\end{proposition}
\noindent  See \cite{War} for the proof.
\vskip 1pc

 A $C^{\infty}$  curve $x(t) = (x_1(t), \cdots, x_N(t))  $ for $t$ in some interval   is an integral curve of $X$ if
$$ \frac{d { x} (t)}{dt} = X( x(t)).$$

  A submanifold $M\subset \Bbb R^N $ is said to be an invariant submanifold of $X$ if $M$ is invariant under the flow of $X$, that  is,
for any integral curve $ x(t) $ with $x (0)\in M$ we have $ x(t) \in M $ for all $t$.
 $M$ is an invariant
submanifold of $X$ if and only if $X$ is tangent to $M$ at every point of $M$.
  An   integral curve is an invariant submanifold of dimension $1$.   Every level set of a non-degenerate
system of
first integrals  $(\r_1, \ldots, \r_d)$ is an invariant submanifold of codimension $d$.    The following  is rather obvious:

\begin{proposition} \label{invariant hyper} Let $\r$    be a  $C^{\infty}$  function with  $ d\r\neq 0  $ defined on an open subset of $\Bbb R^N$ and  $M$ be the zero set of $\r$.  Then  $M$ is an invariant hypersurface of a non-vanishing vector field $X$   if and only if
\begin{equation} \label{generalized}
(X\r) (x) = 0   \quad\text{for all } x \text{ with } \r(x)=0.\end{equation}
\end{proposition}


\section{Solution by means of invariant submanifolds}\label{psec3}
\vskip 1pc
In this section we present an analytic $(C^{\o})$  version of the method of characteristics for quasi-linear equations of first order.   Some relevant results are found in \cite{HK} and \cite{HP}.
 Coming back to the Cauchy problem    (\ref{pde})-(\ref{data})    consider the vector field
\begin{equation}\label{char1}  X:= \a \frac{\partial}{\partial t} +  \sum_{k=1}^n a_k \frac{\partial}{\partial x_k} + b \frac{\partial}{\partial u},
\end{equation}
\noindent which we call the characteristic vector field of (\ref{pde}).   The method of characteristics is to construct  the graph of the solution of (\ref{pde})-(\ref{data}) by finding integral curves of (\ref{char1}) starting from each point of the initial set $\G$.   It is the problem of solving the following system of initial value problems of ordinary differential equations:
\begin{equation}\label{characteristic}\begin{aligned}
\dot{t} &= \a,  ~~ \quad  t(0) = 0 \\
\dot{x_k} &= a_k,  \quad x_k(0)= s_k, \quad  k=1, \cdots, n \\
\dot{u} &= b, \quad u(0) = h(s), ~~ s:= (s_1, \cdots, s_n),
\end{aligned} \end{equation}
where $\dot{x }$ means the derivative of $x$ with respect to the  parametrization   of the  curves.
 The solution of (\ref{characteristic}) is an $n$-parameter family  parameterized by $s=(s_1, \cdots, s_n) $ of integral curves of the characteristic vector field. Therefore,  solving (\ref{characteristic}) is  equivalent to finding a hypersurface $\Sig \subset \Bbb R^{n+2}$ such that

\begin{itemize}
\item[i)] $\G \subset \Sig, $  where $\G = \{(0, s, h(s)): |s|<\e\}$ is the initial set.

\item[ii)] $\Sig$ is invariant under the flow of $X.$
\end{itemize}

 The graph of a solution to (\ref{pde})-(\ref{data}) is part of $\Sig.$  We first prove the
 following
\begin{theorem} \label{inv hyper} Let $F(t,x,u)$ be an analytic  real-valued function on an open set   $\O\subset\Bbb R^{n+2} = \{(t,x,u)\} $
such that $F_u \neq 0.$     Then  $F(t,x,u) = 0 $  is an implicit solution of (\ref{pde}) if and only if the zero set of
 $F$  is  invariant under the flow of  (\ref{char1}).
\end{theorem}

\begin{proof}

Suppose $F(t,x,u) = 0 $ is an implicit solution.  Let $ p=(t_0, x_0, u_0) $ be an arbitrary point of the zero set of $F$.
To show that the zero set of $F$ is invariant under the flow of the characteristic vector $X,$  it is enough  by Proposition \ref{invariant hyper}
 to show that
\begin{equation*} X(p) F = 0.\end{equation*}
Let $u(t,x)$ be as in Theorem \ref{implicit function}, namely, $u(t_0, x_0)=u_0 $ and
\begin{equation}\label{main6} F(t,x,u(t,x)) = 0.\end{equation}
Differentiating (\ref{main6}) with respect to $t$ and $x_k$, respectively,  we have
\begin{equation}\label{main7}
 \begin{aligned}
F_t + F_u u_t  &= 0 \\
F_{x_k} + F_u u_{x_k} &=0, ~~ k=1, \cdots, n.
\end{aligned}\end{equation}
Therefore, \begin{equation}\label{main8} \begin{aligned} X(p)F & =
 \a (p) F_t + \sum_{k=1}^n a_k(p) F_{x_k} + b(p) F_u \\
&= F_u(p) \left( - \a(p) u_t  - \sum_{k=1}^n a_k(p) u_{x_k} + b(p) \right)  \quad \text{by  } (\ref{main7})   \end{aligned}
\end{equation} where the last line is zero because $u(x,t)$ is an explicit solution to (\ref{pde}).

Conversely,  suppose that the zero set of $F$ is invariant under the flow of the characteristic vector (\ref{char1}).
It suffices to show that at an arbitrary point $p=(t_0, x_0, u_0)$ of the zero set of $F$,  the function $u(t,x)$ as in Theorem \ref{implicit function}
with $u(t_0, t_0)=u_0 $  is an explicit solution to (\ref{pde}).  Again, by differentiating $F(t,x,u(t,x)) = 0 $ with respect to $t$ and with respect to $x_k$, we have (\ref{main7}).   Then as in (\ref{main8}) we have
\begin{equation}\label{main9} \begin{aligned} 0 &= X(p) F \\
&=   F_u(p) \left( - \a(p) u_t  - \sum_{k=1}^n a_k(p) u_{x_k} + b(p) \right).   \end{aligned}
\end{equation}
The last line of (\ref{main9}) being zero and $F_u\neq 0 $ imply that $u(t,x)$ is the explicit  solution to (\ref{pde}).
\end{proof}

Now we find  $F(t,x,u)$ that defines an  invariant hypersurface  as  in Theorem \ref{inv hyper}  by  using  the first integrals of (\ref{char1}) as follows:
Let  ${\vec \r} = (\r_1, \cdots, \r_{n+1})$  be a nondegenerate system of  $C^{\o}$   first integrals of (\ref{char1})
defined
on a neighborhood of the initial set $\G$  as in (\ref{gamma1}).

  Consider the mapping
\begin{equation}\label{char9} \Bbb R^{n+2} \ni (t,x_1, \cdots, x_n, u) \stackrel{\vec\r}{\mapsto} (\r_1, \cdots, \r_{n+1}) \in \Bbb R^{n+1}.
\end{equation}
Since $\vec \r$ is nondegenerate, $\vec\r(\Gamma) $ is a $C^{\o}$  hypersurface in $\Bbb R^{n+1}$, thus there exists an analytic  function $f$ so that
$\vec\r(\Gamma)$ is contained in  the zero set of $ f$.  Now let
\begin{equation}\label{char11}  F:= f \circ \vec \rho. \end{equation}
\noindent Then $F$ itself is a first integral of (\ref{char1}) that is non-degenerate and   vanishes on $\Gamma. $
Thus we proved the following
\begin{proposition}\label{main10}
Let $\vec \r :=(\r_1, \cdots, \r_{n+1}) $ be a nondegenerate system of $C^{\o}$ first integrals of the characteristic vector field (\ref{char1})
defined on a neighborhood of the initial set $\G.$   Let $f$ be an analytic  local defining function of the   hypersurface $\vec \r (\G) \subset \Bbb R^{n+1}. $      Then the zero set of
$ F(t,x,u) := f \circ \vec{\r} $  has the  following properties:
 \begin{itemize}

\item[i)]$ F(\G) = 0,$

\item[ii)] the zero set of $F$ is invariant under the flow of (\ref{char1}).
\end{itemize}
\end{proposition}

Now let $\pi: (t,x,u) \mapsto (t,x) $ be the projection.
From the invariant set $F(t,x,u) = 0 $ as in Proposition \ref{main10} we find the maximal domain of analytic extension  of the local explicit solution $u(t,x)$   as follows:   Let $\O\subset \Bbb R^{n+2} $ be the open set where the coefficients
 $\a, $  $ a_k $ and   $b$ of   (\ref{pde}) are defined and
  $\Sig^{n+1} \subset \O $ be the connected component of the zero set of
$ F$ that contains $\G^n,$
and $\sig^n \subset \Sig^{n+1}$ be the set of points where  $ F_u(t,x,u) = 0. $  Then
  $  \Sigma^{n+1}\setminus \sig^n \stackrel{\pi}{ \rightarrow}  \pi(\Sig^{n+1}\setminus \sig^n)$
is a local diffeomorphism by the implicit function theorem. Now let $\Sig_{\G}$ be the connected component of
$\Sig^{n+1}\setminus \sig^n $ that contains $\G^n$.
  We fix a point $P\in \G^n$ and let  $Q \in \Sig^{n+1} $ be any point. If  $Q\in\Sig_{\G}$ then there is a curve in $\Sig_{\G}$ that connects  $P$ to $Q$,  which implies that the local solution $u(t,x)$ to (\ref{pde})-(\ref{data}) analytically continues to $\p(Q)$ by the analytic implicit function theorem.  If $Q\in \Sig^{n+1}\setminus \Sig_{\G}$ then any curve in $\Sig^{n+1}$ that connects $P$ to $Q$ intersects $\sig^n$, which implies that $u(t,x)$ does not continue analytically to $\p(Q)$  because at the intersecting point $|\nabla u|$ blows up by the implicit function theorem. Thus we  proved the following

\begin{theorem}\label{main} Given a quasi-linear PDE of first order (\ref{pde})  where the coefficients $\a, $ $a_j $ and $b$ are analytic on an open set $\O\subset \Bbb R^{n+2} =\{(t,x, u)\}$,  $x=(x_1, \cdots, x_n), $  let $X$ be the characteristic vector field on $\O$ given by (\ref{char1}),
and let  $\G^n=\{(0, x, h(x)):  |x| < \e \}$
be given an initial set   where  $h$ is analytic.
Suppose that there exists a set of first integrals  $\vec{\r} = (\r_1, \cdots, \r_{n+1})$ of $X$ that are defined in $\O$ and nondegenerate   on
a neighborhood of $\G^n$.
Let $f$ be a defining function of  $\vec{\r}(\G^n)$ that is   analytic in  $\vec{\r}(\O)  $
 and let $F = f\circ \vec\r.$ Let
$ \Sigma^{n+1} $ be the connected component of the zero set of $F$ that contains $\G^n$
 and $ \sig^n := \{(t,x,u)\in \Sig^{n+1} | ~ F_u(t,x,u) = 0\} $
and    $\pi:  \Bbb R^{n+2} \rightarrow \Bbb R^{n+1} $ be the projection
$(t,x,u) \stackrel{\pi}{\rightarrow} (t,x). $
 Then  we have
\begin{itemize}

\item[i)]$F(t,x,u) = 0 $ is  an implicit solution to (\ref{pde})-(\ref{data}),

\item[ii)]$  \Sigma^{n+1}\setminus \sig^n \stackrel{\pi}{ \rightarrow}  \pi(\Sig^{n+1}\setminus \sig^n)$
is a local diffeomorphism,

\item[iii)]Let $\Sig_{\G}$ be the connected component of $\Sig^{n+1} \setminus \sigma^n$   that contains $\G^n.$   Then $\Sig_{\G}$  is the graph of the maximal analytic extension of  the local solution
$u(t,x)$ to (\ref{pde})-(\ref{data}),

\item[iv)]$\p(\Sig_{\G})$ is the maximal domain of analytic extension of the local solution to (\ref{pde})-(\ref{data}).
\end{itemize}
\end{theorem}

\begin{example}\label{ODE}   Case $n=0$: This is the case of  ordinary differential equations.     For   a real-valued function $u(t)$ consider
\begin{equation}\label{ODE2} u' = u^2 \end{equation} subject to the initial condition
\begin{equation}\label{ODE3} u(0) = 1.\end{equation}
It is easy to find  the solution  $\displaystyle{ u = \frac{1}{1-t}}, $  so that the maximal domain is $(-\infty, 1).$      Now we shall obtain the same by  our theory and Theorem~\ref{main}.  The characteristic vector field is
\begin{equation*} X = \frac{\pa}{\pa t} + u^2 \frac{\pa}{\pa u}
\end{equation*}  so that
\begin{equation*}
\r := t+\frac{1}{u}
\end{equation*} is a first integral.   On the initial data set $\G^0 = \{(0,1)\} $  $\r$ has value $1$, so that $f (y) = y-1$ and
\begin{equation*}\label{ODE4} F := f \circ \r =  \r -1 = t + \frac{1}{u} - 1.
\end{equation*}   Now  $ F = 0 $ is an implicit solution to (\ref{ODE2})-(\ref{ODE3})
and the connected component of $\G^0$  in $F=0$ is the set
$$\Sig^1 = \{(t,u)\in \Bbb R^2  ~:~ t+\frac{1}{u} - 1 = 0, ~ t<1\}.$$
We see further that $\sig^0  $ is empty,  which implies  $\Sig_{\G} = \Sig^1 $,  and that $\pi(\Sig_{\G}) =(-\infty, 1).$
\end{example}

\begin{example} \label{example1} Case $n=1.$   For a real-valued function $u (t,x) $ of two real variables  consider
\begin{equation}\label{exeq1} u u_t = -t \end{equation} subject to the initial condition
\begin{equation}\label{initial1}  u(0,x)  = \sqrt{1-x^3}, \quad | x| < \epsilon.
\end{equation}

Then the set of initial data is $\G^1 = \{(0,x, \sqrt{1-x^3})~:~ |x| < \e\} $ and the
 characteristic vector field is
\begin{equation}\label{char2}
X:= u \frac{\partial }{\partial t} - t \frac{\partial }{\partial u}.
\end{equation}
Since $X$ has no $\frac{\pa}{\pa x}$ component,  $\r_1(t,x,u):= x $ is  obviously a first integral.  It is easy to find another  first integral $\r_2(t,x,u) := t^2+ u^2.$
This implies  that each characteristic curve is a circle
 $ t^2+u^2 = c_1 $ and  $  x = c_2  $ for constants  $c_1 $ and $ c_2, $
  so that the graph of the solution is  contained
in  a surface of revolution obtained by rotating about $x$-axis.  Now let  $$\vec \r = (\r_1, \r_2): \Bbb R^3 \longrightarrow \Bbb R^2.$$
Since $\vec{\r}(\G^1) = \{(x, 1-x^3)\},  $
$$f (\r_1, \r_2):= \r_2 -1+(\r_1)^3 = 0 \quad \text{on  }  \G^1.$$ Hence

\begin{equation}\label{exeq2} F:= f (\vec \r) =   t^2 + u^2 - 1 + x^3  = 0
\end{equation} is the implicit  solution of (\ref{exeq1}).
 Let $\Sig^2$ be the surface defined by  (\ref{exeq2}).
To find the  set  $\sigma^1$ of the singular points solve  (\ref{exeq2}) together with
\begin{equation}\label{exeq3}  F_u = 2u = 0
\end{equation}   simultaneously.  We have
$$\sigma^1 = \{(t,x,u): u=0,~~ x^3+t^2-1=0\}.$$
  Therefore, the maximal domain of analytic extension  is
\begin{equation}\label{max dom}  1-t^2-x^3 > 0, \end{equation}
 as shown in Figure~\ref{fig1_1} and the maximally extended single-valued solution is
\begin{equation}\label{exeq11} u(t,x) = (1-t^2-x^3)^{1/2}. \end{equation}

\begin{figure}
  \includegraphics[scale=0.4]{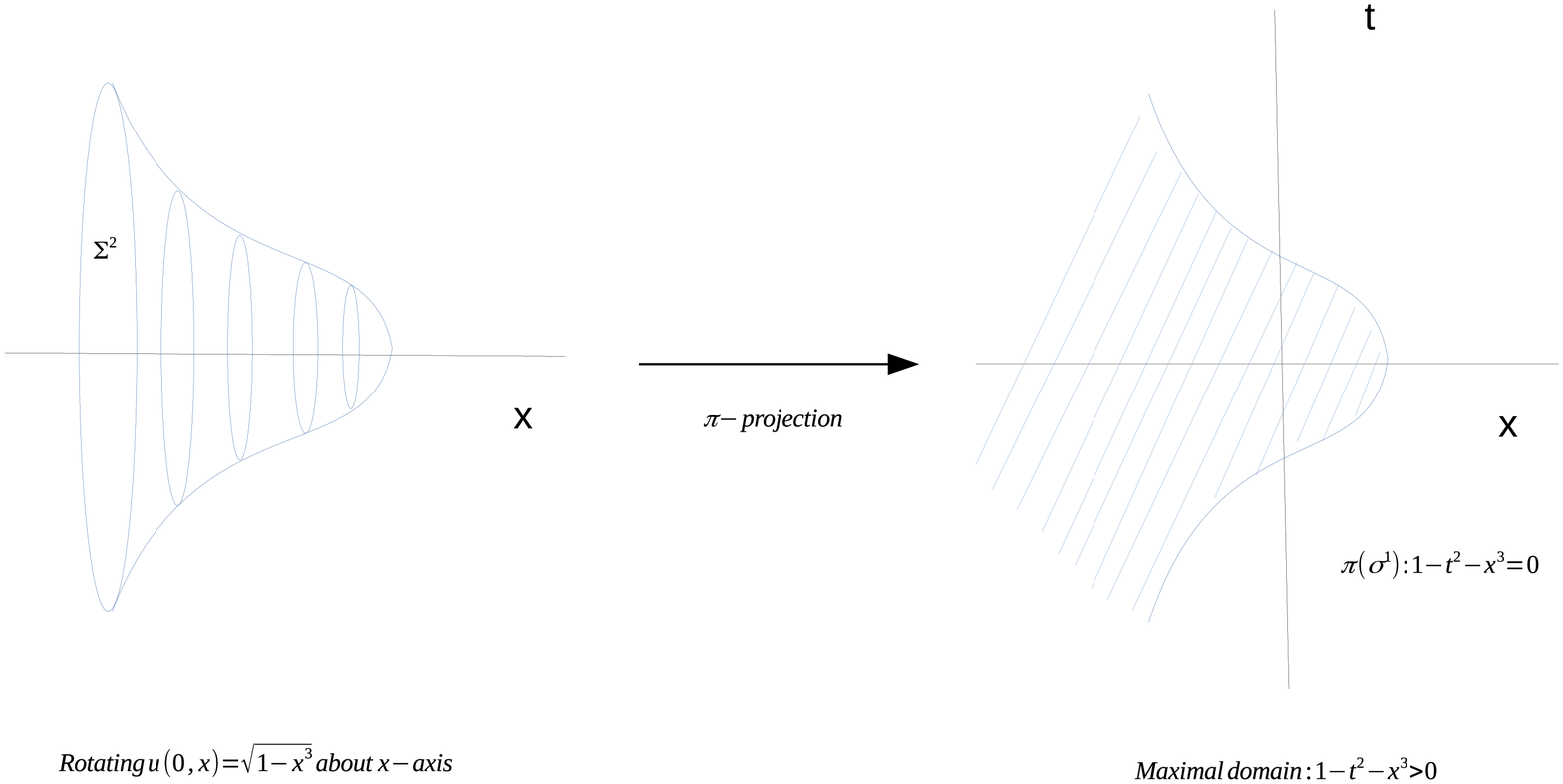}
  \caption{}
  \label{fig1_1}
  \end{figure}
\end{example}

\section{Conservation laws}

A quasi-linear PDE of first order for $u(t,x)$ of the form
\begin{equation*} \label{conserv}  u_t + \Phi(u)_x = 0. \end{equation*} is called a {\it $1$-dimensional conservation law}.  This describes the motion of
 $1$-dimensional flow of fluid  where  $u$ is the conserved density  and $\Phi(u)$ is the flux.
Let us consider the case $\Phi(u) = u^2/2$ so that  the equation becomes
\begin{equation}\label{conserv4}  u_t + u u_x  = 0, \end{equation} which is the inviscid Burgers' equation.  We assume (\ref{conserv4}) holds for all time $-\infty < t < \infty. $
 The  characteristic vector field  of (\ref{conserv4}) is
\begin{equation}\label{conserv5}   X = \frac{\pa}{\pa t} + u \frac{\pa}{\pa x}. \end{equation}
Observe that $X$ has no $\frac{\pa}{\pa u}$ component and   thus $\r_1(t,x,u) := u $ is a first integral.  We see that
$\r_2(t,x,u) := x - ut $ is another first integral and that ${\vec \r} = (\r_1, \r_2)$ is non-degenerate and globally defined.
Now we consider the Cauchy problem (\ref{conserv4}) with the initial condition
$u(0,x) = h(x): ~|x|<\e.$   It is well known that if $h$ is strictly decreasing then the solution $u(t,x)$ has singularities for some positive $t$  (See \cite{CH}, \cite{Evans}).  In this case we observe that the boundary of the maximal domain is the envelope of the projections of the  characteristic lines
(See Example~\ref{example linear ic} and Example~\ref{example2}).

\begin{example} \label{example linear ic} Consider (\ref{conserv4}) with the initial condition
\begin{equation*}\label{linear initial condition} u(0,x) =  - kx, ~ |x|<\e
\end{equation*} where $ k>0 $ is a  constant.
As in Example~\ref{example1} and in Example~\ref{example2} we have $F(t,x,u) = (1-kt)u+kx $  and  $\sig^1$  is the line $t=1/k, x=0$ that is parallel to $u$-axis. The maximal domain of analytic extension of solution is the half plane $\{(t,x):~~~ t<1/k\}.$
\end{example}

\begin{example}\label{example2}  We consider the same equation (\ref{conserv4}) with a decreasing initial data
\begin{equation}\label{conserv10}    u(0,x) = \frac{1}{x+1}, \quad |x|<\e.
\end{equation}
  Let  $\G^1 = \{(0,x,\frac{1}{x+1}):  |x|<\e\} $  be the initial set. Then
$$ \vec{\r}(\G^1) = \{(1 /(x+1), x): |x|<\e\} $$
is given by   $$f(\vec\r) := \r_1 - \frac{1}{\r_2 + 1} = 0.$$  Therefore,  the implicit solution to (\ref{conserv4}) is
\begin{equation} \label{conserv6} F(t,x,u): = f \circ \vec\r = u - \frac{1}{x-ut+1} = 0  \end{equation}
and (\ref{conserv6}) defines $\Sig^2.~ $~   Solving (\ref{conserv6})  for $u$ we have
\begin{equation}\label{algebraic solution} u=
\begin{cases} \begin{aligned}&\frac{x+1 \pm \sqrt{(x+1)^2-4t}}{2t} ~~ &\text{if } t\neq 0\\
&\frac{1}{x+1} 
& \text {if } t=0.\end{aligned}\end{cases}
\end{equation}

To  figure out the shape of $\Sig^2$ we recall that  the quantities $u$ and $x-ut $  remain  constant along the integral curves of (\ref{conserv5}).
These curves are  given by (\ref{characteristic}) with $n=1,$  $\a=1, $ $a=u, $   $b=0, $ $h(s)= \frac{1}{s+1},  $   namely,
     \begin{equation*}\label{characteristic lines} \begin{cases} \dot{t}=1,\quad  t(0)=0\cr
\dot{x} =u, \quad x(0)=s\cr
\dot{u} = 0, \quad u(0)=\frac{1}{s+1}, \end{cases}\end{equation*} which give $1$-parameter family of lines
  \begin{equation}\label{characteristic lines2}
x = s+ut, \quad
u = \frac{1}{s+1}.\end{equation}
Given each value $s,$ (\ref{characteristic lines2}) is the  {\it characteristic line} for (\ref{conserv4}) through $(0, s, \frac{1}{s+1})\in \G^1.$
Projections to the $(t,x)$-plane of some characteristic lines are shown in Figure~\ref{fig3_3}.  On the other hand, by solving  (\ref{conserv6}) and
\begin{equation}
F_u  =  1 - \frac{t}{(x-ut+1)^2} = 0
\end{equation} simultaneously,  we find  that  $ \sigma^1 $ is the curve  $t = \frac{(x+1)^2}{4}, $  $u = \frac{2}{x+1}, $
 so that   $\pi(\sigma^1) $
 is the curve
 \begin{equation} \label{sing2}  t = \frac{(x+1)^2}{4}. \end{equation}    Easy computation shows that for each $s$ the line  (\ref{characteristic lines2}) is tangent to (\ref{sing2}) at the point
\begin{equation}\label{tangent point}
t=(s+1)^2, ~~ x=2(s+1)-1. \end{equation}   Observe  that   $\p(\sig^1)$ is part of the envelope of the projections to $(t,x)$-plane of  (\ref{characteristic lines2}) with $s\neq -1.$
The maximal domain of analytic extension is the shaded area of  Figure~\ref{fig3_3}, which  is the union of characteristic projections with positive slope.
In order for the analytic continuation to be single-valued we take each characteristic projection up to the point (\ref{tangent point}).
\begin{figure}
  \includegraphics[scale=0.4]{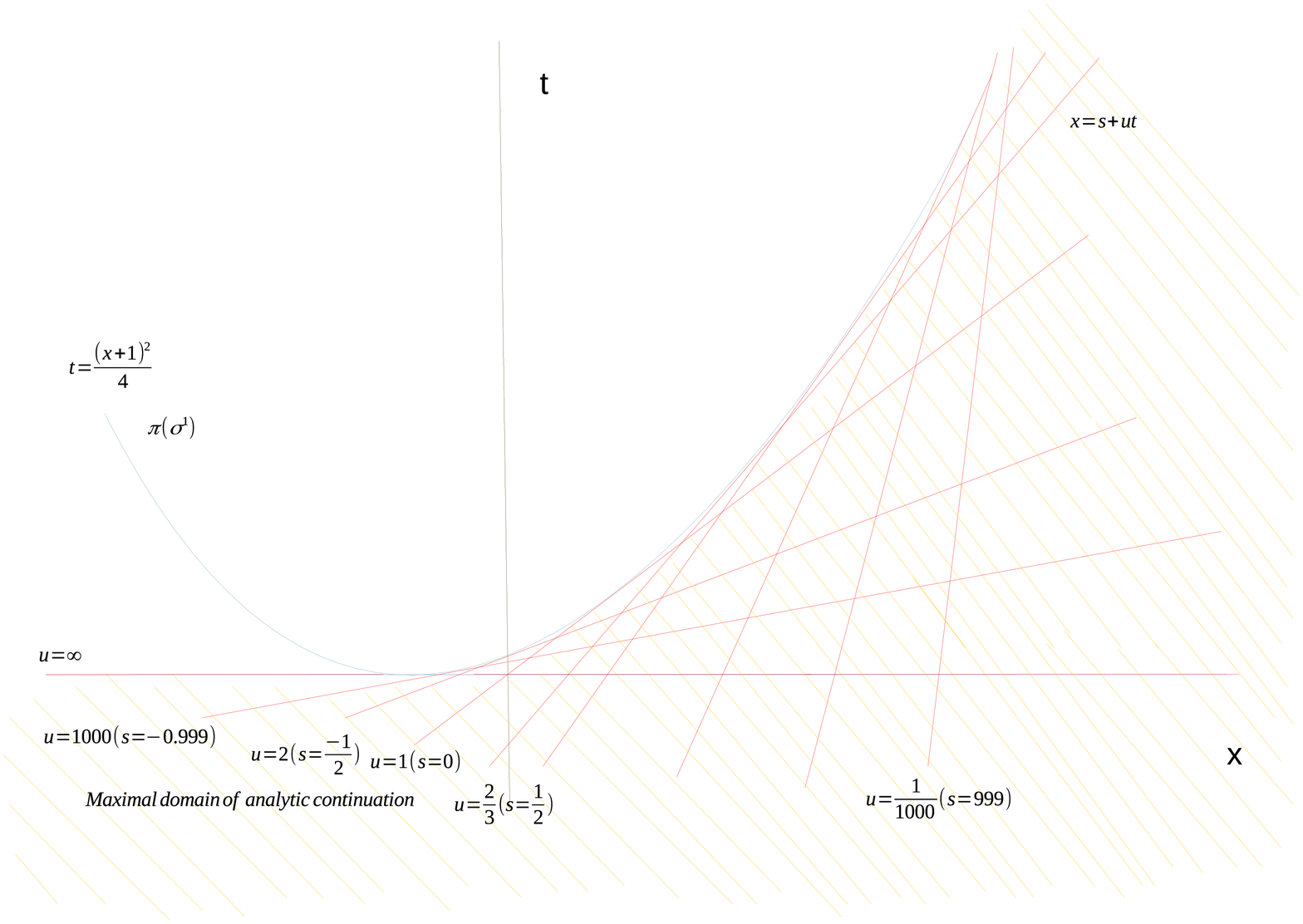}
   \caption{}
   \label{fig3_3}
  \end{figure}
By differentiating (\ref{sing2}) with respect to $t$ we obtain
\begin{equation*}\label{speed} \frac{dx}{dt} = \frac{1}{\sqrt{t}} = \frac{2}{x+1},\end{equation*}
which gives an alternative way of  finding   the speed of propagation of singularity  for analytic cases.
As for the singularities of solutions of $1$-D conservation laws we refer the readers to
 \cite{lax}.
 \end{example}

\end{document}